\documentclass[10pt, reqno]{amsart}
\usepackage{amscd,amssymb, amsmath, wasysym}
\usepackage{graphicx}
\usepackage{amsfonts}
\usepackage{mathrsfs}    
\usepackage{amsmath}    
\usepackage{amsthm}     
\usepackage{amscd}      
\usepackage{amssymb}    
\usepackage{latexsym}   
\usepackage{graphicx}   
\usepackage{verbatim}   
\usepackage[all]{xy}     
\usepackage{xfrac}
\usepackage{amsmath,amssymb,fancyhdr,color,epsfig}

\usepackage[margin=3cm]{geometry}

\usepackage{color}
\usepackage[usenames,dvipsnames,svgnames,table]{xcolor}

\usepackage{hyperref}
\hypersetup{
	colorlinks   = true,
	citecolor    = NavyBlue,
	linkcolor = magenta
	}

\usepackage[T2A]{fontenc}
\usepackage[utf8]{inputenc}
\usepackage[english]{babel}

\usepackage{textcase}
\usepackage{wrapfig}

\newtheorem{theorem}[subsection]{Theorem}
\newtheorem{lemma}[subsection]{Lemma}

\newtheorem{proposition}[subsection]{Proposition}

\newtheorem{remark}[subsection]{Remark}
\newtheorem*{remark*}{Remark}

\makeatletter
\@addtoreset{equation}{section}
\@addtoreset{figure}{section}
\@addtoreset{table}{section}
\makeatother

\makeatletter
\newcommand\RRR{\mathbb{R}}
\newcommand\ZZZ{\mathbb{Z}}
\newcommand\FFF{\mathcal{F}}
\newcommand\id{\mathrm{id}}

\begin{document}
\title[Deformations of smooth function on $2$-torus whose KR-graph is a tree]
{Deformations of smooth function on $2$-torus whose KR-graph is a tree}
\author{Bohdan Feshchenko}
\address{Department of Algebra and Topology, Institute of Mathematics of NAS of Ukraine, Tereshchenkivska str. 3, Kyiv, 01601, Ukraine}
\curraddr{}
\email{fb@imath.kiev.ua}


\begin{abstract}
	Let $f:T^2\to \RRR$ be Morse function on $2$-torus $T^2,$ and $\mathcal{O}(f)$ be the orbit of $f$ with respect to the right action of the group of diffeomorphisms $\mathcal{D}(T^2)$ on $C^{\infty}(T^2)$. Let also $\mathcal{O}_f(f,X)$ be a connected component of $\mathcal{O}(f,X)$ which contains $f.$
	In the case when Kronrod-Reeb graph of $f$ is a tree we obtain the full description of $\pi_1\mathcal{O}_f(f).$
	
	This result also holds for more general class of smooth functions $f:T^2\to \RRR$ which have the following property: for each critical point $z$ of $f$ the germ $f$ of $z$ is smoothly equivalent to some homogeneous polynomial $\RRR^2\to \RRR$ without multiple points
	
\end{abstract}

\maketitle

\section{Introduction}
Let $M$ be a smooth compact surface, $X$ be a closed (possible empty) subset of $M$, $\mathcal{D}(M, X)$ be the group of diffeomorphisms of $M$ fixed on some neighborhood of  $X$. Then the group $\mathcal{D}(M,X)$ acts on the space of smooth functions $C^{\infty}(M)$ on $M$ by the following rule:
\begin{equation}\label{main-act}
\gamma: C^{\infty}(M)\times \mathcal{D}(M,X)\to C^{\infty}(M),\qquad \gamma(f,h) = f\circ h.
\end{equation}
For $f\in C^{\infty}(M)$, let $$\mathcal{S}(f,X) := \{h\in \mathcal{D}(M,X)\,|\, f\circ h = f\},\qquad \text{and}\qquad \mathcal{O}(f,X) := \{f\circ h\,|\, h\in\mathcal{D}(M,X)\}$$ be respectively the stabilizer and the orbit of $f$ with respect to the action $\gamma.$ 

If $X$ is the empty set, then we put
$$
\mathcal{D}(M):=\mathcal{D}(M,\varnothing),\qquad \mathcal{S}(f):=\mathcal{S}(f,\varnothing),\qquad \mathcal{O}(f):=\mathcal{O}(f,\varnothing),
$$
and so on. Endow spaces $\mathcal{D}(M,X)$ and $C^{\infty}(M)$ with the corresponding Whitney topologies; these topologies induce certain topologies on $\mathcal{S}(f,X)$ and $\mathcal{O}(f,X).$

Let $\mathcal{S}_{\id}(f,X)$ and $\mathcal{D}_{\id}(f,X)$ be respectively connected components of $\mathcal{S}(f,X)$ and $\mathcal{D}(f,X)$ which contain the identity map $\id_M,$ and $\mathcal{O}_f(f,X)$ be the connected component of $\mathcal{O}(f,X)$ containing $f.$ We also set $\mathcal{S}'(f,X) = \mathcal{S}(f)\cap \mathcal{D}_{\id}(M,X).$

Let $\mathcal{F}(M)\subset C^{\infty}(M)$ be a set of smooth functions satisfying the following conditions:
\begin{itemize}
	\item[(B)] the function $f$ takes the constant value at each connected component of the boundary $\partial M$, and all critical points of $f$ belongs to the interior of $M$;
	\item[(P)] for each critical point $z$ of $f$ the germ of $f$ in $z$ is smoothly equivalent to some homogeneous polynomial $f_z:\RRR^2\to \RRR$ without multiple factors.
\end{itemize}
Let $\mathrm{Morse}_{\partial}(M)$ be the set of Morse functions satisfying condition (B). It is known that the set $\mathrm{Morse}_{\partial}(M)$ is a dense subset of $C^{\infty}(M).$ By Morse Lemma, every non-degenerate singularity of Morse function is smoothly equivalent to the polynomial $\pm x^2 \pm y^2$ in some chart representation around this critical point. Thus, we have the inclusion $\mathrm{Morse}_{\partial}(M)\subset \mathcal{F}(M)\subset C^{\infty}(M).$

\begin{theorem}[\cite{Maksymenko:AGAG:2006, Maksymenko:ProcIM:ENG:2010, Maksymenko:UMZ:ENG:2012}]
	Let $f\in\mathcal{F}(M)$ and $X$ be a finite (possible empty) union of regular components of level-sets of $f$. Then the following statements hold:
	\begin{enumerate}
		\item the map $p:\mathcal{D}(M,X)\to \mathcal{O}(f,X)$, $p(h) = f\circ h$ is a Serre fibration with the fiber $\mathcal{S}(f,X);$
		\item the restriction $p|_{\mathcal{D}_{\id}(M,X)}:\mathcal{D}_{\id}(M,X)\to \mathcal{O}_f(f,X)$ of the map $p$ is also a Serre fibration; 
		\item consider that $X = \varnothing$ and either $f$ has the critical point which is not non-degenerate local extremum, or $M$ is non-orientable. Then $\mathcal{S}_{\id}(f)$ is contractible, $\pi_n\mathcal{O}_f(f)\cong \pi_n M,$ for $n\geq 3$, $\pi_2\mathcal{O}_f(f) = 0$, and for $\pi_1\mathcal{O}_f(f)$ we have the following sequence
		\begin{equation}\label{eq:hom-main-seq}
		1\longrightarrow \pi_1\mathcal{D}_{\id}(M)\xrightarrow{~~p~~}\pi_1\mathcal{O}_f(f) \xrightarrow{~~\partial~~}\pi_0\mathcal{S}'(f)\longrightarrow 1;
		\end{equation}
		\item consider that $\chi(M) < 0$ or $X\neq  \varnothing$. Then  $\mathcal{D}_{\id}(M,X)$ and $\mathcal{S}_{\id}(f,X)$ are contractible, $\pi_n\mathcal{O}_f(f,X) = 0$ for $n\geq 2$, and the map $\partial: \pi_1\mathcal{O}_f(f,X)\to \pi_0\mathcal{S}'(f,X)$ is the isomorphism.
	\end{enumerate}
\end{theorem}

In the serious of papers \cite{Maksymenko:AGAG:2006, Maksymenko:MFAT:2010, Maksymenko:ProcIM:ENG:2010, Maksymenko:UMZ:ENG:2012, Maksymenko:DefFuncI:2014, Maksymenko:orbfin:2014, Maksymenko:pi1Repr:2014}
Maksymenko described homotopy type of stabilizers of the action (\ref{main-act}).

Maksymenko and the author in papers
\cite{MaksymenkoFeshchenko:MS:2014,MaksymenkoFeshchenko:MFAT:2015} describes fundamental group of the orbit $\pi_1\mathcal{O}_f(f)$ of the function  $f$ in $\FFF(T^2)$ in the case when KR-graph of the function $f$ has a cycle.
In the case when KR-graph of $f$ is a tree, the authors
\cite{MaksymenkoFeshchenko:UMZ:ENG:2014} found conditions under which the sequence (\ref{eq:hom-main-seq}) splits.
The aim of the paper is to describe the group $\pi_1\mathcal{O}_f(f)$ of the  function $f$ from $\FFF(T^2)$ whose KR-graph is a tree.

\begin{remark}
Let also $w:(I^k,\partial I^k, 0)\to (\mathcal{D}_{\id}(M,X), \mathcal{S}_{\id}(f,X),\id_{M})$ be a continuous map of triples, $k\geq 0.$ From (2) of Theorem 1.1 follows that for every $k\geq 0$ there is an isomorphism
$$
\lambda_k:\pi_k (\mathcal{D}_{\id}(M,X), \mathcal{S}_{\id}(f,X),\id_{M})\to \pi_k\mathcal{O}_f(f,X),\qquad \lambda_k([\omega]) = [f\circ \omega],
$$ 
see for example \cite[\S~4.1, Theorem~4.1]{Hatcher:AlgTop:2002}. In the text we will identify $\pi_1\mathcal{O}_f(f)$ with $\pi_1(\mathcal{D}_{\id}(M),\mathcal{S}'(f)).$ 
\end{remark}

\subsection*{Structure of the paper} In Section 2 we collects some preliminaries: wreath products of special type, needed for formulation of the result, Kronrod-Reeb graph of smooth functions and its automorphism group. The main result of the paper is Theorem \ref{thm:main} which is proved in Section 3.

\section{Preliminaries}
\subsection{Wreath products $G\wr_{\ZZZ_n\times \ZZZ_m}\ZZZ^2$} Let $G$ be a group, and $1$ be the unit of $G,$ $\mathrm{Map}(\ZZZ_n\times \ZZZ_m,G)$ be the group of all maps $\ZZZ_n\times \ZZZ_m\to G$ with the point-wise multiplication, i.e.,
for $\alpha,\beta:\ZZZ_n\times \ZZZ_m\to G$ from $\mathrm{Map}(\ZZZ_n\times\ZZZ_m,G)$ we have $(\alpha\beta)(i,j) = \alpha(i,j)\beta(i,j),$ where $(i,j)\in\ZZZ_n\times \ZZZ_m,$ $n,m \geq 1.$ 

The group $\ZZZ^2$ acts from the right on $\mathrm{Map}(\ZZZ_n\times \ZZZ_m, G)$ be the following rule: if $\alpha\in \mathrm{Map}(\ZZZ_n\times \ZZZ_m, G)$ and $(k,j)\in\ZZZ^2$, then the result of this action $\alpha^{k,l}$ is given by 
$$
\alpha^{k,j}(i,j) = \alpha(i + k\,\mathrm{mod}\, n, j + l\,\mathrm{mod}\, m),\qquad (i,j)\in\ZZZ^2.
$$
The semi-direct product $\mathrm{Map}(\ZZZ_n\times \ZZZ_m, G)\rtimes \ZZZ^2,$ which corresponds to this $\ZZZ^2$-action we denote by 
$$
G\wr_{\ZZZ_n\times \ZZZ_m}\ZZZ^2:=\mathrm{Map}(\ZZZ_n\times \ZZZ_m, G)\rtimes \ZZZ^2,
$$
and will call the wreath product of $G$ and $\ZZZ^2$ under $\ZZZ_n\times \ZZZ_m.$

Thus $G\wr_{\ZZZ_n\times \ZZZ_m}\ZZZ^2$ is a Cartesian product with the operation:
$$
(\alpha, (k_1,k_2))(\beta, (l_1,l_2)) = (\alpha\beta^{k_1,k_2}, (k_1+l_1, k_2 + l_2))
$$
for all $(\alpha, (k_1,k_2)), (\beta, (l_1ml_2))\in \mathrm{Map}(\ZZZ_n\times \ZZZ_m, G)\times \ZZZ^2.$ Moreover we have the following exact sequence
$$
1\longrightarrow \mathrm{Map}(\ZZZ_n\times \ZZZ_m, G)\xrightarrow{~~\sigma~~} G\wr_{\ZZZ_n\times \ZZZ_m}\ZZZ^2 \xrightarrow{~~p~~} \ZZZ^2\longrightarrow 1,
$$
where $\sigma(\alpha) = (\alpha, (0,0))$ is the inclusion, and $p(\alpha, (a_1,a_2)) = (a_1,a_2)$ is the projection.

\subsection{Kronrod-Reeb graph of smooth functions}
Let $f\in\mathcal{F}(M)$ and $c$ be a real number. Recall that a connected component $C$ of level-set $f^{-1}(c)$ is called {\it critical} if $C$ contains at almost one critical point of $f$, otherwise $C$ is called {\it regular}.

Let $\Delta$ be a partition of $M$ onto connected components of level-sets of the function $f$. It is well known that the quotient-space $M/\Delta$ is $1$-dimensional CW complex called Kronrod-Reeb graph of $f$, or simply, KR-graph of  $f$; we will denote it by $\Gamma_f$. From the definition of $\Gamma_f$ follows that vertices of $\Gamma_f$ are critical components of level-sets of $f.$ Let also $p_f:M\to \Gamma_f$ be the projection map. So, the function $f$ can be represented as the following composition:
$$
f = \phi_f\circ p_f: M\xrightarrow{~~p_f~~}\Gamma_f\xrightarrow{~~\phi_f~~}\RRR,
$$
where $\phi_f$ is the map induced by $f$. Note that $\phi_f$ is monotone on the set of edges of $\Gamma_f.$

\subsection{Actions of $\mathcal{S}(f)$ on $\Gamma_f$}
Let $f$ be a smooth function from $\FFF(M),$ and $h\in\mathcal{S}'(f)$. Then $f\circ h = f$ by definition of $h$, and hence, $h(f^{-1}(c)) = f^{-1}(c)$ for all $c\in\RRR.$ Then $h\in\mathcal{S}'(f)$ interchanges level-sets of $f.$  So, $h$ induces the homomorphism $\rho(h)$ of KR-graph $\Gamma_f$ such that the diagram
$$
\xymatrix{
	M\ar[r]^{p_f} \ar[d]^{h} & \Gamma_f \ar[r]^{\phi_f}\ar[d]^{\rho(h)} & \RRR \ar@{=}[d]\\
	M\ar[r]^{p_f} & \Gamma_f \ar[r]^{\phi_f} & \RRR
	}
$$
is commutative. In other words, we have a homomorphism:
$
\rho:\mathcal{S}'(f)\to \mathrm{Aut}(\Gamma_f).
$
Denote by $G$ the image of $\mathcal{S}'(f)$ in $\mathrm{Aut}(\Gamma_f)$ under the map $\rho.$

Let $v$ be the vertex of $\Gamma_f$ and $G_v = \{g\in G\,|\, g(v) = v\}$ be a stabilizer of $v$ under the action $G$. By star $\mathrm{st}(v)$ of the vertex $v$ we will mean a closed connected $G_v$-invariant neighborhood of $v$ in $\Gamma_f$, which does not contains the other vertices of $\Gamma_f$.  The set $G_v^{loc} = \{g|_{\mathrm{st}(v)}\,|\, g\in G_v\}$ is a subgroup of $\mathrm{Aut}(\mathrm{st}(v))$, which consists of the restrictions of elements of $G_v$ onto $\mathrm{st}(v).$ We will call it {\it the local stabilizer} of $v$ under the action of $G.$ Note that the group $G_{v}^{loc}$ does not depends on the choice of  the star $\mathrm{st}(v)$ of the vertex $v$. In particular, the following diagram 
$$
\xymatrix{
	\mathcal{S}'(f) \ar[d]_{pr} \ar[r]^{\rho} & G \ar[d]_r \ar[r] & \mathrm{Aut}(\Gamma_f)\\
	\pi_0\mathcal{S}'(f) \ar[ru]^{\rho_0}\ar[r]^{\widehat{\rho}} & G_v^{loc} \ar[r] & \mathrm{Aut}(\mathrm{st}(v))
	}
$$
is commutative, where $p$ is a projection, $r$ is the restriction map onto $\mathrm{st}(v),$ $\rho = \rho_0\circ r$, and $\widehat{\rho} = r\circ \rho.$

For the function $f\in\FFF(T^2)$ on $2$-torus the following result holds:
\begin{lemma}[Proposition 1, \cite{MaksymenkoFeshchenko:UMZ:ENG:2014}]\label{lm:unique_vertex}
	Let $f\in\FFF(T^2)$ be such that its KR-graph $\Gamma_f$ is a tree. Then there exists the unique vertex $v$ of the graph $\Gamma_f$ such that each component of $T^2 - p_f^{-1}(v)$ is an open $2$-disk.
\end{lemma} 
The vertex $v$ of $\Gamma_f$ and the critical component $V = p_f^{-1}$ of $f^{-1}(\phi_f(v))$, which corresponds to $v$ we will call {\it special}.

The main result of out paper is the following result:
\begin{theorem}\label{thm:main}
	Let $f\in\FFF(T^2)$ be such that $\Gamma_f$ is a tree, and $v$ be the special vertex of $\Gamma_f$. Then 
	\begin{enumerate}
		\item $G_v^{loc}\cong \ZZZ_n\times \ZZZ_{nm}$ for some $n,m\in\mathbb{N}$;
		\item there exist closed $2$-disks $D_1,D_2,\ldots, D_r\subset T^2$ such that $f|_{D_i}\in\FFF(D_i),$ $i = 1,2,\ldots, r$, and there is an isomorphism
		$$
		\xi:\pi_1\mathcal{O}_f(f)\cong \prod_{i = 1}^r \pi_0\mathcal{S}'(f|_{D_i},\partial D_i)\wr_{\ZZZ_n\times \ZZZ_{nm}}\ZZZ^2.
		$$
	\end{enumerate}
\end{theorem}
In particular, in the case $G_v^{loc} = 1$ we have the isomorphism $\xi:\pi_1\mathcal{O}_f(f)\cong \pi_0\mathcal{S}'(f)\times \ZZZ^2,$ see Theorem 2 \cite{MaksymenkoFeshchenko:UMZ:ENG:2014}.

\subsection{Combinatorial actions of finite groups on surfaces}
Now we recall some results from \cite{Feshchenko:2016:ActTree}.
Let $f\in\FFF(M)$. Suppose that its Kronrod-Reeb graph $\Gamma_f$ contains a special vertex $v$, and $V$ be the special component of level set of $f$ which corresponds to $v.$

Let $\mathcal{S}_V(f) = \{h\in \mathcal{S}(f)\;|\;h(V) = V\}$ be a subgroup of $\mathcal{S}(f)$ leaving $V$ invariant.  It is easy to see that $\rho(\mathcal{S}_V(f))\subset G_v.$ We denote by $\phi$ the  map
$$
\phi = r\circ \rho: \mathcal{S}_V(f)\xrightarrow{~~\rho~~} G_v\xrightarrow{~~r~~} G_v^{loc}.
$$

Let $H$ be a subgroup of $G_v^{loc}$ and $\mathcal{H}= \phi^{-1}(H)$ be a subgroup of $\mathcal{S}_V(f).$ We will say that the group
$\mathcal{H}$ has property {\rm (C)}  if the following conditions hold.
\begin{itemize}
	\item[(C)] Let $h\in\mathcal{H},$ and $E$  be a $2$-dimensional element of $\Xi$. Suppose that $h(E) = E.$ Then $h(e) = e$ for all other $e\in\Xi$, and the map $h$ preserves orientation of each element of $\Xi.$
	
\end{itemize}

\begin{proposition}[Theorem 2.2 \cite{Feshchenko:2016:ActTree}]\label{th:com-act-tree}
	Suppose $f\in\FFF(M)$ is  such that its KR-graph $\Gamma_f$ contains a special vertex $v$, and $G_v^{loc}$ be the local stabilizer of $v.$ Let also $H$ be a subgroup of $G_v^{loc},$ and $\mathcal{H} = \phi^{-1}(H)$ be a subgroup of $\mathcal{S}_V(f)$ satisfying condition $(\mathrm{C}).$ Then there exists a section $s:H\to\mathcal{H}$ of the map $\phi,$ i.e., the map $s$ is a homomorphism satisfying the  condition $\phi\circ s = \id_{H}.$
\end{proposition}

\section{Proof of statement (1) of Theorem \ref{thm:main}}
Let $f\in\FFF(T^2)$ be such that its KR-graph $\Gamma_f$ is a tree, and $v$ is the special vertex of the graph $\Gamma_f$, and $V$ be the connected component of $f$, which corresponds to $v.$ We need to show that $G_v^{loc} \cong \ZZZ_n\times \ZZZ_{nm}$ for some $n,m\in\mathbb{N}.$

Note that from Lemma 2.4 follows that $V$ gives the partition of $T^2$: $0$-dimensional and $1$-dimensional cells are vertices and edges of $V$ respectively, and $2$-dimensional cells are connected component of $T^2\setminus V.$

From \cite[Theorem~7.1]{Maksymenko:AGAG:2006} follows that for each $h\in\ker(r\circ \rho)$ the following conditions hold:
\begin{enumerate}
	\item $h(e) = e$ for each cell $e$,
	\item the map $h:e\mapsto h(e)$ preserves orientations of cells $e$ of dimension $\dim e \geq 1.$
\end{enumerate}
Let $h\in\mathcal{S}'(f)$ be a diffeomorphism. According to \cite[Proposition~5.4]{Maksymenko:ProcIM:ENG:2010} either all cells are $h$-invariant, or the number of invariant cells under $h$ is equal to Lefschetz number $L(h)$. Since $h$ is isotopic to the identity map, it follows that $L(h) = \chi(T^2) = 0$. Thus, <<combinatorial>> action of $h$ on the set of all cells defines by the action of $h$ on any fixed $2$-cell, i.e., by the action $\rho(h)$ on the edge of $\mathrm{st}(v)$. Therefore, from  Proposition \ref{th:com-act-tree} follows that there exists the section $s: G_v^{loc}\to \mathcal{S}'(f)$ of the map $r\circ \rho$ such that $s(G_v^{loc})$ freely acts on $T^2.$ In particular, the quotient-map $q:T^2\to T^2/G_{v}^{loc}$is the covering map, and hence, $T^2/G_v^{loc}$ is  either $2$-torus $T^2$, or Klein bottle. But since $G_v^{loc}$-action on $T^2$ is the action by diffeomorphisms which preserve orientation and are isotopic to $\id_{T^2}$, it follows that the quotient-map $T^2/G_v^{loc}$ is a torus. In particular, we have the following short exact sequence:
$$
1\longrightarrow \pi_1 T^2\xrightarrow{~~q~~} \pi_1 T^2/G_v^{loc}\longrightarrow G_v^{loc} \longrightarrow 1.
$$
Since $q$ is the monomorphism, it follows that the proposition (1) of Theorem \ref{thm:main} is the consequence of the following result:
\begin{lemma}[Chapter E, \cite{Pont}]
	Let $A,B$ be free abelian group of the rank $2$, and $q: A\to B$ be an inclusion. Then there exists $\mathsf{L},\mathsf{M}\in A$ and $X, Y\in B$ such that $A = \langle \mathsf{L},\mathsf{M}\rangle,$ $B = \lambda X,Y,\rangle$, and 
	$$
	q(\mathsf{L}) = nX,\qquad q(\mathsf{M}) = mnY
	$$
	for some $n,m\in\mathbb{N}$, in particular $B/A \cong \ZZZ_n\times \ZZZ_{mn}.$
\end{lemma} 

\section{Proof of statement (2) of Theorem \ref{thm:main}}
\subsection*{Step 1. Choice of generators of $\pi_1 T^2$ and $\pi_1 T^2/G_v^{loc}$} 
Fix a point $y\in T^2$. Let $z = q(y)\in T^2/G_v^{loc}$. Then we have the following diagram:
$$
\xymatrix{
	0 \ar[r] & \pi_1 (T^2, y)\ar@{=}[d] \ar[r]^q & \pi_1(T^2/G_v^{loc}, z)  \ar@{=}[d]\ar[r]^{\partial} & G_v^{loc} \ar@{=}[d]\ar[r] & 0\\
	0\ar[r] & \ZZZ^2\ar[r]^q & \ZZZ^2 \ar[r]^{\partial} & \ZZZ_n\times \ZZZ_{nm} \ar[r] & 0
	}
$$
where $q: \ZZZ^2\to \ZZZ^2$ and $\partial:\ZZZ^2\to \ZZZ_n\times \ZZZ_{nm}$ are defined by the formulas:
$$
q(\lambda,\mu) = (n\lambda, mn\mu),\qquad \partial(x,y) = (x\,\mathrm{mod}\, n, y\,\mathrm{mod}\, nm).
$$

Let $\widehat{X},\widehat{Y}:T^2/G_v^{loc}\times [0,1]\to T^2/G_v^{loc}$ be isotopies such that $\widehat{X}_0 = \widehat{X}_1 = \widehat{Y}_0 = \widehat{Y}_1 = \id_{T^2/G_v^{loc}}$, and $\widehat{X}_s\circ \widehat{Y}_t = \widehat{Y}_t\circ \widehat{X}_s$ for all $s,t \in[0,1]$. Moreover loops $\widehat{X}_z, \widehat{Y}_z:I\to T^2/G_v^{loc}$, defined by  $\widehat{X}_z(t) = \widehat{X}(z,t)$ and $\widehat{Y}_z(t) = \widehat{Y}(z,t)$, represent elements $[\widehat{X}_z] = (1,0)$ and $[\widehat{Y}_z] = (0,1)$ in $\pi_1(T^2/G_v^{loc},z).$

Extend $\widehat{X}$ and $\widehat{Y}$ to maps $X,Y:T^2/G_v^{loc}\times \RRR\to T^2/G_v^{loc}$ by formulas:
$$
X(x,t) = \widehat{X}(x,t\,\mathrm{mod}\,1),\qquad Y(x,t) = \widehat{Y}(x, t\,\mathrm{mod}\, 1).
$$

Let $\mathsf{L},\mathsf{M}: T^2\times \RRR\to T^2$ be the unique lifftings of $X$ and $Y$ respectively with respect to the map $q$ such that $\mathsf{L}$ and $\mathsf{M}$ commutes and $\mathsf{L}_0 = \mathsf{M}_0 = \id_{T^2},$ i.e., $X_t\circ q = q\circ L_t$ and $Y_t\circ q = q\circ \mathsf{M}_t$.

Let $s: G_v^{loc}\to \mathcal{S}'(f)$ be a section of  the map  $r\circ \rho,$ see Proposition \ref{th:com-act-tree}. Since by (1) of Theorem \ref{thm:main}, $G_v^{loc}\cong \ZZZ_n\times \ZZZ_{mn}$ for some $n,m\in\mathbb{N}$, it follows that $\mathsf{L}_t\circ \mathsf{M}_{t'} = \mathsf{M}_{t'}\circ \mathsf{L}_t$ for all $t, t'\in\RRR$ and 
$$
\mathsf{L}_k = s(k\,\mathrm{mod}\,n,0),\qquad \mathsf{M}_k = s(0,k\,\mathrm{mod}\, mn)
$$
for all $k\in\ZZZ.$ In particular $\mathsf{L}_{kn} = \mathsf{M}_{kmn} = \id_{T^2},$ $k\in\ZZZ^2,$ and loops $\mathsf{L}_z:[0,n]\to T^2$ and $\mathsf{M}_z:[0,mn]\to T^2$ represent elements $[\mathsf{L}_z] = (1,0)$ and $[\mathsf{M}_z] = (0,1)$ in $\pi_1(T^2,y)\cong \ZZZ^2.$

Since $G_v^{loc}$ freely acts on $T^2$, it follows that connected components of $\overline{T^2 - V}$ we can enumerate by three indexes $D_{ijk}$ such that $i = 1,2,\ldots, r,$ $j = 0,1,\ldots, n-1$, and $k = 0,1,\ldots, mn-1$. Moreover if $\gamma = (a,b)\in\ZZZ_n\times\ZZZ_{mn}\cong G_v^{loc},$ then 
$$
\gamma(D_{ijk}) = D_{i,j+a, k+b},
$$ 
where second index takes modulo $n$, and third by modulo $mn.$

We put $\mathcal{S}_{ijk} = \pi_0\mathcal{S}'(f|_{D_{ijk}},\partial D_{ijk})$ and $\mathcal{S} = \prod_{i = 1}^{r}\prod_{j = 0}^{n-1}\prod_{k = 0}^{mn-1}\mathcal{S}_{ijk}$. Next, define the homomorphism 
$$
\tau:\mathcal{S}\to \mathrm{Map}(G_v^{loc}, \prod_{i = 1}^r \mathcal{S}_{i00})
$$
by the formula: if $\alpha = (h_{ijk})\in \mathcal{S}$, then the map
$$
\tau(\alpha):\ZZZ_n\times \ZZZ_{mn}\to \prod_{i= 1}^r\mathcal{S}_{i00}
$$
is given by the formula
\begin{equation}
\tau(\alpha)(a,b) = (\mathsf{M}_k^{-1}\circ \mathsf{L}_j^{-1}\circ h_{ijk}\circ \mathsf{L}_j\circ \mathsf{M}_k,\; i = 1,2,\ldots, r),
\end{equation}
where $(a,b)\in \ZZZ_n\times \ZZZ_{mn}\cong G_v^{loc}.$ After direct verifying we see that $\tau$ is the isomorphism.

\subsection*{Step 2. Epimorphism $\psi$} Let $h:I\to \mathcal{D}_{\id}(T^2)$ be a loop in $\mathcal{D}_{\id}(T^2)$ such that $h(0) = h(1) = \id_{T^2},$ i.e., $h$ is an isotopy $h: T^2\times I\to T^2$ of the torus. Let $x$ be a point in $T^2.$ Then $h_x:\{x\}\times I\to T^2$ is a loop in $T^2$ with the starting point $x$. Define the map $\ell:\pi_1\mathcal{D}_{\id}(T^2)\to\pi_1T^2$ by the formula: $\ell([h]) = [h_x]\in \pi_1T^2.$
It is known that the map $\ell$ is the isomorphism, see \cite{EarleEells:DG:1970,EarleSchatz:DG:1970,Gramain:ASENS:1973}.

\begin{lemma}\label{lm:psi}
	There exists the epimorphism $\psi:\pi_1(\mathcal{D}_{\id}(T^2),\mathcal{S}'(f))\to \pi_1 T^2/G_v^{loc}$ such that the following diagram is commutative
	\begin{equation}\label{diag:psi}
	\xymatrix{
		1\ar[r] & \pi_1\mathcal{D}_{\id}(T^2) \ar[d]_{\ell}^{\cong} \ar[r] & \pi_1(\mathcal{D}_{\id}(T^2),\mathcal{S}'(f)) \ar[d]^{\psi} \ar[r] & \pi_0\mathcal{S}'(f) \ar[d]^{\widehat{\rho}_0} \ar[r] & 1\\
		1\ar[r] & \pi_1 T^2 \ar[r]^q & \pi_1 T^2/G_v^{loc} \ar[r] & G_v^{loc} \ar[r] & 1
		}
	\end{equation}
	and rows are exact sequences.
\end{lemma}
\begin{proof}
	Fix any vertex $z$ of $V$ and define the map $\psi_0:\mathcal{D}_{\id}(T^2)\to T^2/G_v^{loc}$ by $\psi_0(h) = q(h(z)),$ where $h\in\mathcal{D}_{\id}(T^2)$ and $q:T^2\to T^2/G_v^{loc}$ is a quotient-map induced by free action of $G_v^{loc}$ on $T^2$. Obviously, $\psi_0$ is continuous map. Since $G_v^{loc}$-action  and $\mathcal{S}'(f)$-action coincide on vertices of $V,$ it follows that $\psi_0(h)$ belongs to some $G_v^{loc}$-orbit of the point $z$ for $h\in\mathcal{S}'(f)$. Then the map $\psi_0$ induces the map of triple
	$$
	\psi_0: (\mathcal{D}_{\id}(T^2),\mathcal{S}'(f),\id_{T^2})\to (T^2/G_v^{loc},z,z),\qquad \psi_0(\widehat{h}) = q(\widehat{h}(z)).
	$$
	In particular, $\phi_0$ induces the homomorphism
	$$
	\phi: \pi_1(\mathcal{D}_{\id}(T^2),\mathcal{S}'(f),\id_{T^2}) \to \pi_1 (T^2/G_v^{loc},z,z)
	$$
	Since rows in diagram \ref{diag:psi} are exact sequences, the map $\ell$ is the isomorphism, the map $\widehat{\rho}_0$ is the epimorphism, it follows that, by 5-lemma, the map $\phi$ is the epimorphism.
\end{proof}

\subsection*{Step 3. Kernel of $\psi$} Let $f(V) = c,$ $\epsilon > 0$ and $N$ be a connected component of $f^{-1}([c-\epsilon, c+ \epsilon])$ contains $V.$ We will call $N$ the $f$-regular neighborhood of $V$. Recall that $\mathcal{S}'(f,N):=\{h\in \mathcal{S}'(f)\,|\, h|_{N} = \id_{N}$.

The following lemma describes the kernel of $\psi.$
\begin{lemma}
	There exist isomorphisms between these five groups
	$$
	\ker\psi \xrightarrow{~~\zeta~~} \ker \widehat{\rho}_0 \xleftarrow{~~\iota~~}\pi_0\mathcal{S}'(f, N)\xrightarrow{~~\sigma~~}\mathcal{S}\xrightarrow{~~\tau~~}\mathrm{Map}(G_v^{loc}, \prod_{i = 1}^r\mathcal{S}_{i00})
	$$
\end{lemma}
\begin{proof}
	(1) First we construct the isomorphism $\zeta:\ker\psi\to \ker\widehat{\rho}_0.$ Consider the following diagram with exact rows and columns:
	$$
	\xymatrix{
	&& 1 \ar[d] & 1 \ar[d] & \\
	&& \pi_1\mathcal{D}_{\id}(T^2) \ar[d] \ar[r]^{\ell}_{\cong} & \pi_1 T^2 \ar[d] & 	\\
	1 \ar[r] & \ker\psi  \ar[d]_{\zeta}^{\cong} \ar[r] & \pi_1(\mathcal{D}_{\id}(T^2),\mathcal{S}'(f)) \ar[d]_{\partial\circ \lambda_1^{-1}} \ar[r] & \pi_1T^2/G_v^{loc} \ar[r] \ar[d] & 1\\
	1 \ar[r] & \ker\widehat{\rho}_0 \ar[r] & \pi_0\mathcal{S}'(f)  \ar[r]^{\widehat{\rho}_0} \ar[d] & G_v^{loc}  \ar[d] \ar[r] & 1\\
	&& 1  & 1  & 
		}
	$$
	Since $\ell$ is the isomorphism, it follows that by $3\times 3$-lemma, the homomorphism $\zeta = \partial\circ \lambda_1^{-1}|_{\ker\psi}$ is the isomorphism.
	
	(2) Note that the following map is the isomorphism
	$$
	\sigma:\mathcal{S}'(f,N)\cong \prod_{i,j,k}\mathcal{S}'(f|_{D_{ijk}},\partial D_{ijk}),\qquad  \sigma(h) = (h|_{D_{ijk}})_{ijk},
	$$
	which induces the an isomorphism
	$$
	\sigma:\pi_0\mathcal{S}'(f,N)\cong \prod_{i = 1}^r\prod_{j = 0}^{n-1}\prod_{k = 0}^{nm-1}\mathcal{S}_{ijk} = \mathcal{S}.
	$$
	
	(3) It sufficient to show that the inclusion $\iota:\mathcal{S}'(f,N)\to \ker(r\circ \rho)$ is a homotopy equivalence. Hence it induces the isomorphism $\iota:\pi_0\mathcal{S}'(f,N)\to\pi_0\ker(r\circ\rho) = \ker\widehat{\rho}_0.$
	Now we show that there exists an isotopy $H:\ker(r\circ \rho)\times I\to \ker(r\circ\rho)$ such that the following conditions hold:
	\begin{itemize}
		\item[(i)] $H_0 = \id_{T^2};$
		\item[(ii)] $H_t(\mathcal{S}'(f,N))\subset \mathcal{S}'(f,N)$ for all $t\in I;$
		\item[(iii)] $H_1(\ker(r\circ \rho))\subset \mathcal{S}'(f,N).$
	\end{itemize}
	
	Let $F$ be the Hamiltonian vector field of the function $f\in\FFF(T^2),$ $\mathbf{F}:T^2\times \RRR\to T^2$ be the flow of $F,$ and $N$, $N'$ be $f$-regular neighborhoods of $V$ such that $\overline{N}\subset \mathrm{Int}(N').$ For each smooth function $\gamma:T^2\to \RRR$ define the map $\mathbf{F}_{\gamma}: T^2\to T^2$ by the formula $\mathbf{F}_{\gamma}(x) = \mathbf{F}(x,\gamma(x)).$
	
	From \cite[Claim 1]{Maksymenko:ProcIM:ENG:2010} follows that for each $h\in\ker(r\circ \rho)$ there exists a unique smooth function $\beta_h\in C^{\infty}(N')$ such that $h = \mathbf{F}_{\beta_h}$ on $N',$ i.e., $h(x) = \mathbf{F}(x,\beta_h(x)),$ $x\in N'.$ Moreover the map $\widehat{s}:\ker(r\circ \rho)\to C^{\infty}(N')$ defined by $\widehat{s}(h) = \beta_h$ is continuous map with respect $C^{\infty}$-topologies. Furthermore, if $h$ is fixed on $N,$ then $\beta_h = 0$ on $N.$
	
	Extend the function $\beta_h$ to the smooth function $\alpha_h\in C^{\infty}(M)$ such that $\alpha_h|_{N} = \beta_h$ and $\alpha_h = 0$ on $T^2-N'$ in the following way. Let $\varepsilon: T^2\to [0,1]$ be a smooth function on $T^2$ such that 
	\begin{enumerate}
		\item $\varepsilon$ is constant on orbits of flow $\mathbf{F}$;
		\item $\varepsilon = 1$ on $N$;
		\item $\varepsilon = 0$ on $T^2-N'.$
	\end{enumerate}
	Define $\alpha_h = \varepsilon\beta_h$ on $N'$ and $\alpha_h = 0$ on $T^2-N'.$ Obviously, the correspondence $h\mapsto \alpha_h$ is the continuous map $\alpha:\ker(r\circ \rho)\to C^{\infty}(T^2).$ Furthermore from condition (1) on the function $\varepsilon$ follows that the map $\mathbf{F}_{t\alpha_h}:T^2\to T^2$ defined by the formula $\mathbf{F}_{t\alpha_h}(x) = \mathbf{F}(x,t\alpha_h(x))$ is the diffeomorphism for all $t\in I$, see \cite[Claim~4.14.1]{Maksymenko:AGAG:2006}. From conditions (2) and (3) we have that
	$$
	\mathbf{F}(x,\alpha_h(x)) = \begin{cases}
	h(x), & x\in N,\\
	x, & x\in T^2 - N'.
	\end{cases}
	$$
	Define the isotopy $H:\ker(r\circ \rho)\times I \to \ker(r\circ \rho)$ by the formula $H(x, t) = h\circ \mathbf{F}^{-1}_{t\alpha_h}.$ It remains to prove that $H$ satisfies conditions (i)-(iii). Indeed,
	
	(i) $H_0(h) = h\circ \mathbf{F}_0^{-1} = h,$ i.e., $H_0 = \id|_{\ker(r\circ \rho)};$
	
 	(ii) Consider that $h\in \mathcal{S}'(f,N)$. Then $\beta_h = t\alpha_t = 0$ in $N$, and hence, $\mathbf{F}_{t\alpha_h}|_{N} = \id_N$ for all $t\in I.$  In particular $H_t(h)|_{N} = h|_N = \id_N.$
 	
 	(iii) $H_1(h)|_{N} = h\circ \mathbf{F}^{-1}|_{N} = h\circ h^{-1}|_{N} = \id_N.$\\
 	So lemma is proved.
\end{proof}

\subsection*{Step 4. Defining the map $\xi$} Define the map 
$$
\xi:\mathrm{Map}(G_v^{loc}, \prod_{i = 1}^r \mathcal{S}_{i00})\rtimes \pi_1(T^2/G_v^{loc})\to \pi_1(\mathcal{D}_{\id}(T^2),\mathcal{S}'(f),\id_{T^2})
$$
by the following way. Let $\alpha:\ZZZ_n\times \ZZZ_{mn}\cong G_v^{loc}\to \prod_{i = 1}^r\mathcal{S}_{i00}$ be any map. For each triple $(i,j,k)$ we chose $h_{ijk}\in\mathcal{S}'(f|_{D_{i00}},\partial D_{i00})$ such that 
$$
\alpha(i,j) = ([h_{1jk}], [h_{2jk}],\ldots, [h_{rjk}]),
$$
and let $h_{ijk}^t: D_{i00}\to D_{i00}$ be any isotopy between $h^0_{ijk}= \id_{D_{i00}}$ and $h^1_{ijk} = h_{ijk}$. Define the map 
$$
h:(I,\partial I,0)\to (\mathcal{D}_{\id}(T^2),\mathcal{S}'(f),\id_{T^2})
$$ 
by the formula:
$$
h(t)(x) = \begin{cases}
\mathsf{M}_{k+at}\circ \mathsf{L}_{j+bt}\circ h^t_{ijk}\circ \mathsf{L}_{j}^{-1}\circ \mathsf{M}_{k}^{-1}(x), & x\in D_{ijk},\\
\mathsf{M}_{at}\circ \mathsf{L}_{bt}(x), & x \in N. 
\end{cases}
$$
It is easy to see that $h$ is well defined. We set
$$
\xi(\alpha, (a,b)) = [h]\in\pi_1(\mathcal{D}_{\id}(T^2),\mathcal{S}'(f),\id_{T^2}).
$$
Also it is no difficult to verify that  the map $\xi$ is the homomorphism. Furthermore from Lemma \ref{lm:psi} and the formula for the map $\tau$ follows that the following diagram is commutative: 
$$
\xymatrix{
	1 \ar[d] &&& 1\ar[d]\\
	\mathrm{Map}(G_v^{loc},\prod_{i = 1}^r\mathcal{S}_{i00}) \ar[rrr]^{(\tau\circ\sigma\circ \iota^{-1}\circ\zeta^{-1})^{-1}}_{\cong} \ar[d] &&& \ker\psi \ar[d]\\
	\mathrm{Map}(G_v^{loc},\prod_{i = 1}^r\mathcal{S}_{i00})\rtimes \pi_1 T^2/G_v^{loc} \ar[d]^{pr}  \ar[rrr]^{\xi}&&& \pi_1(\mathcal{D}_{\id}(T^2),\mathcal{S}'(f)) \ar[d]\\
	\pi_1T^2/G_v^{loc} \ar[d]\ar@{=}[rrr] &&& \pi_1 T^2/G_v^{loc} \ar[d] \\
	1 &&& 1
	}
$$
By 5-lemma we have that $\xi$ is the isomorphism. Theorem \ref{thm:main} was proved.

{\bf Acknowledgment.} The author would like to thank S.  Maksymenko for  attention to my work.

\bibliographystyle{plain}

\end{document}